\documentclass{amsart}
\usepackage{amsfonts}
\usepackage{hyperref}

\newtheorem{theorem}{Theorem}
\theoremstyle{plain}

\newtheorem{corollary}{Corollary}

\newtheorem{lemma}{Lemma}

\newtheorem{proposition}{Proposition}
\newtheorem{remark}{Remark}

\numberwithin{equation}{section}
\input{tcilatex}

\begin{document}
\title[Some sharp Wilker type inequalities]{Some sharp Wilker type
inequalities and their applications}
\author{Zheng-Hang Yang}
\address{Power Supply Service Center, Zhejiang Electric Power Corporation
Research Institute, Hangzhou City, Zhejiang Province, 310009, China}
\email{yzhkm@163.com}
\date{April 6, 2013}
\subjclass[2010]{Primary 26D05, 33B10; Secondary 26A48, 26D15, }
\keywords{Wilker inequality, trigonometric function, hyperbolic function,
mean}
\thanks{This paper is in final form and no version of it will be submitted
for publication elsewhere.}

\begin{abstract}
In this paper, we prove that for fixed $k\geq 1$, the Wilker type inequality 
\begin{equation*}
\frac{2}{k+2}\left( \frac{\sin x}{x}\right) ^{kp}+\frac{k}{k+2}\left( \frac{%
\tan x}{x}\right) ^{p}>1
\end{equation*}%
holds for $x\in \left( 0,\pi /2\right) $ if and only if $p>0$ or $p\leq -%
\frac{\ln \left( k+2\right) -\ln 2}{k\left( \ln \pi -\ln 2\right) }$. It is
reversed if and only if $-\frac{12}{5\left( k+2\right) }\leq p<0$. Its
hyperbolic version holds for $x\in \left( 0,\infty \right) $ if and only if $%
p>0$ or $p\leq -\frac{12}{5\left( k+2\right) }$. And, for fixed $k<-2$, the
hyperbolic version is reversed if and only if $p<0$ or $p\geq -\frac{12}{%
5\left( k+2\right) }$. Our results unify and generalize some known ones.

.
\end{abstract}

\maketitle

\section{Introduction}

Wilker \cite{Wilker.AMM.96.1.1989} proposed two open problems, the first of
which states that if $x\in \left( 0,\pi /2\right) $ then 
\begin{equation}
\left( \frac{\sin x}{x}\right) ^{2}+\frac{\tan x}{x}>2,  \label{W}
\end{equation}%
which was proved by Sumner et al. in \cite{Sumner.AMM.98.1991}.

Wilker inequality (\ref{W}) and the second one have attracted great interest
of many mathematicians and have produced a batch of Wilker type ones by
various generalizing and improving as well as different methods and ideas
(see \cite{Baricz.JMI.2.3.2008}, \cite{Chen.in press}, \cite%
{Chen.JIA.accepted}, \cite{Mortitc.MIA.14.3.2011}, \cite%
{Neuman.MIA.13.4.2010}, \cite{Neuman.MIA.15.2.2012}, \cite{Wu.39(2009)}, 
\cite{Wu.AML.2011}, \cite{Wu.ITSF.18.2007}, \cite{Wu.ITSF.19.2008}, \cite%
{Zhang.MIA.11.2008}, \cite{Zhu.MIA.10.2007}, \cite{Zhu.AAA.2009}, \cite%
{Zhu.CMA.58.2009}, \cite{Zhu.JIA.2010.130821} and related references
therein).

In \cite{Wu.ITSF.18.2007}, Wu and Srivastava established another Wilker type
inequality%
\begin{equation}
\left( \frac{x}{\sin x}\right) ^{2}+\frac{x}{\tan x}>2\text{ \ for }x\in
\left( 0,\pi /2\right) ,  \label{Wu1}
\end{equation}%
and proved a weighted and exponential generalization of Wilker inequality.

\noindent \textbf{Theorem Wu (\cite[Theorem 1]{Wu.ITSF.18.2007}).} \emph{Let 
}$\lambda >0,%
\mu
>0$\emph{\ and }$p\leq 2q%
\mu
/\lambda $.\emph{\ If }$q>0$\emph{\ or }$q\leq \min \left( -1,-\lambda /%
\mu
\right) $\emph{, then }%
\begin{equation}
\frac{\lambda }{\lambda +%
\mu
}\left( \frac{\sin x}{x}\right) ^{p}+\frac{%
\mu
}{\lambda +%
\mu
}\left( \frac{\tan x}{x}\right) ^{q}>1  \label{Wu2}
\end{equation}%
\emph{holds for }$x\in \left( 0,\pi /2\right) $\emph{.}

As an application of the inequality (\ref{Wu2}), an open problem posed by
the S\'{a}dor--Bencze in \cite{Sandor.RGMIA.8.3.2005} was solved and
improved. Recently, the inequality (\ref{Wu2}) and all results in \cite%
{Wu.ITSF.18.2007} were extended in \cite{Baricz.JMI.2.3.2008} to Bessel
functions. A hyperbolic version of Theorem Wu has been presented in \cite%
{Wu.AML.2011} very recently.

In 2009, Zhu \cite{Zhu.AAA.2009}\textbf{\ }gave another exponential
generalization of Wilker inequality (\ref{W}) as follows.

\noindent \textbf{Theorem Zh1 (\cite[Theorem 1.1, 1.2]{Zhu.AAA.2009}).} 
\emph{Let }$0<x<\pi /2$\emph{. Then the inequalities}%
\begin{equation}
\left( \frac{\sin x}{x}\right) ^{2p}+\left( \frac{\tan x}{x}\right)
^{p}>\left( \frac{x}{\sin x}\right) ^{2p}+\left( \frac{x}{\tan x}\right)
^{p}>2  \label{Zhwt}
\end{equation}%
\emph{hold if }$p\geq 1$,\emph{\ while the first one in (\ref{Zhwt}) holds
if and only if }$p>0.$

\noindent \textbf{Theorem Zh2 (\cite[Theorem 1.3, 1.4]{Zhu.AAA.2009}).} 
\emph{Let }$x>0$\emph{. Then the inequalities}%
\begin{equation}
\left( \frac{\sinh x}{x}\right) ^{2p}+\left( \frac{\tanh x}{x}\right)
^{p}>\left( \frac{x}{\sinh x}\right) ^{2p}+\left( \frac{x}{\tanh x}\right)
^{p}>2  \label{Zhwh}
\end{equation}%
\emph{hold if }$p\geq 1$\emph{, while the first one in (\ref{Zhwh}) holds if
and only if }$p>0.$

In the end of the same paper, Zhu posed two open problems: find the
respective largest range of $p$ such that the inequalities (\ref{Zhwt}) and (%
\ref{Zhwh}) hold. They have been solved by Mateji\v{c}ka in \cite%
{Matejicka.IJOPCM.4.1.2011}.

Another inequality associated with Wilker one is the following 
\begin{equation}
2\frac{\sin x}{x}+\frac{\tan x}{x}>3  \label{H}
\end{equation}%
for $x\in \left( 0,\pi /2\right) $, which is known as Huygens inequality 
\cite{Huygens.1888-1940}. The following refinement of Huyegens inequality is
due to Neuman and S\'{a}ndor \cite{Neuman.MIA.13.4.2010}:%
\begin{equation}
2\frac{\sin x}{x}+\frac{\tan x}{x}>2\frac{x}{\sin x}+\frac{x}{\tan x}>3,
\label{Nueman}
\end{equation}%
where $x\in \left( 0,\pi /2\right) $. Very recently, the generalizations of (%
\ref{Nueman}), similar to (\ref{Zhwt}), has been derived by Neuman in \cite%
{Neuman.MIA.13.4.2010}. In \cite{Zhu.CMA.58.2009.a}, Zhu proved that for $%
x\in \left( 0,\pi /2\right) $ 
\begin{eqnarray}
\left( 1-\xi _{1}\right) \frac{\sin x}{x}+\xi _{1}\frac{\tan x}{x}
&>&1>\left( 1-\eta _{1}\right) \frac{\sin x}{x}+\eta _{1}\frac{\tan x}{x},
\label{Zhth1} \\
\left( 1-\xi _{2}\right) \frac{x}{\sin x}+\xi _{2}\frac{x}{\tan x}
&>&1>\left( 1-\eta _{2}\right) \frac{x}{\sin x}+\eta _{2}\frac{x}{\tan x}
\label{Zhth2}
\end{eqnarray}%
with the best constants $\xi _{1}=1/3$, $\eta _{1}=0$, $\xi _{2}=1/3$, $\eta
_{2}=1-2/\pi $. Later, he in \cite{Zhu.CMA.58.2009} generalized inequalities
(\ref{Zhth1}) and (\ref{Zhth2}) in exponential form, which is stated as
follows.

\noindent \textbf{Theorem Zh3 (\cite[Theorem 1.1, 1.2]{Zhu.AAA.2009})}. 
\emph{Let }$0<x<\pi /2$\emph{. Then we have}

\emph{(i) when }$p\geq 1$\emph{, the double inequality\ }%
\begin{equation}
\left( 1-\lambda \right) \left( \frac{x}{\sin x}\right) ^{p}+\lambda \left( 
\frac{x}{\tan x}\right) ^{p}<1<\left( 1-\eta \right) \left( \frac{x}{\sin x}%
\right) ^{p}+\eta \left( \frac{x}{\tan x}\right) ^{p}  \label{Zhth3}
\end{equation}%
\emph{holds if and only if }$\eta \leq 1/3$\emph{\ and }$\lambda \geq
1-\left( 2/\pi \right) ^{p}$\emph{.}

\emph{(ii) when }$0\leq p\leq 4/5$\emph{, the double inequality (\ref{Zhth3}%
) holds if and only if }$\lambda \geq 1/3$\emph{\ and }$\eta \leq 1-\left(
2/\pi \right) ^{p}$\emph{.}

\emph{(iii) when \thinspace }$p<0$\emph{, the second one in (\ref{Zhth3})
holds if and only if }$\eta \geq 1/3$\emph{.}

The hyperbolic version of inequalities (\ref{Nueman}) was given in \cite%
{Neuman.MIA.13.4.2010} by Neuman and S\'{a}ndor. In the same year, Zhu
showed that

\noindent \textbf{Theorem Zh4 (\cite[Theorem 4.1]{Zhu.JIA.2010.130821})}. 
\emph{Let }$x>0$\emph{. Then }

\emph{(i) when }$p\geq 4/5$\emph{, the double inequality}%
\begin{equation}
\left( 1-\lambda \right) \left( \frac{x}{\sinh x}\right) ^{p}+\lambda \left( 
\frac{x}{\tanh x}\right) ^{p}<1<\left( 1-\eta \right) \left( \frac{x}{\sinh x%
}\right) ^{p}+\eta \left( \frac{x}{\tanh x}\right) ^{p}  \label{Zhhh1}
\end{equation}%
\emph{holds if and only if }$\eta \geq 1/3$\emph{\ and }$\lambda \leq 0$%
\emph{;}

\emph{(ii) when }$p<0$\emph{, the inequality}%
\begin{equation}
\left( 1-\eta \right) \left( \frac{x}{\sinh x}\right) ^{p}+\eta \left( \frac{%
x}{\tanh x}\right) ^{p}>1  \label{Zhhh2}
\end{equation}%
\emph{holds if and only if }$\eta \leq 1/3$\emph{.}

The aim of this paper is to find the best $p$ such that the inequalities

\begin{eqnarray}
\frac{2}{k+2}\left( \frac{\sin x}{x}\right) ^{kp}+\frac{k}{k+2}\left( \frac{%
\tan x}{x}\right) ^{p} &>&1\text{ for\ }x\in \left( 0,\pi /2\right) ,
\label{Mt} \\
\frac{2}{k+2}\left( \frac{\sinh x}{x}\right) ^{kp}+\frac{k}{k+2}\left( \frac{%
\tanh x}{x}\right) ^{p} &>&1\text{ for }x\in \left( 0,\infty \right)
\label{Mh}
\end{eqnarray}%
or their reverse ones hold for certain fixed $k$ with $k\left( k+2\right)
\neq 0$. In Section 2, some useful lemmas are proved. necessary and
sufficient conditions for (\ref{Mt}) or its reverse and (\ref{Mh}) to hold
are presented in Section 3. Some applications of our main results given in
Section 4.

\section{Lemmas}

The following two lemmas is very important in the sequel.

\begin{lemma}
\label{Lemma AB/C}Let $A$, $B$ and $C$ be defined on $\left( 0,\pi /2\right) 
$ by 
\begin{eqnarray}
A &=&A\left( x\right) =\left( \cos x\right) \left( \sin x-x\cos x\right)
^{2}\left( x-\cos x\sin x\right) ,  \label{A} \\
B &=&B\left( x\right) =\left( x-\cos x\sin x\right) ^{2}\left( \sin x-x\cos
x\right) ,  \label{B} \\
C &=&C\left( x\right) =x\left( \sin ^{2}x\right) \left( -2x^{2}\cos x+x\sin
x+\cos x\sin ^{2}x\right) .  \label{C}
\end{eqnarray}%
Then for fixed $k\geq 1$ the $x\mapsto C\left( x\right) /\left( kA\left(
x\right) +B\left( x\right) \right) $ is increasing on $\left( 0,\pi
/2\right) $. Moreover, we have%
\begin{equation}
\tfrac{5}{12\left( k+2\right) }<\dfrac{C\left( x\right) }{kA\left( x\right)
+B\left( x\right) }<1.  \label{RAB/C}
\end{equation}
\end{lemma}

\begin{proof}
Evidently, $A$, $B>0$ for $x\in \left( 0,\pi /2\right) $ due to $\left( \sin
x-x\cos x\right) >0$ and $\left( x-\cos x\sin x\right) =\left( 2x-\sin
2x\right) /2>0$, while $C>0$ because 
\begin{equation*}
\left( -2x^{2}\cos x+x\sin x+\cos x\sin ^{2}x\right) =x^{2}\left( \cos
x\right) \left( \left( \frac{\sin x}{x}\right) ^{2}+\frac{\tan x}{x}%
-2\right) >0
\end{equation*}%
by Wilker inequality (\ref{W}).

Denote $\left( kA+B\right) /C$ by $D$ and factoring yields 
\begin{eqnarray*}
D\left( x\right) &=&\tfrac{x\left( \sin ^{2}x\right) \left( -2x^{2}\cos
x+x\sin x+\cos x\sin ^{2}x\right) }{\left( \sin x-x\cos x\right) \left(
x-\cos x\sin x\right) \left( \left( 1-k\cos ^{2}x\right) x+\left( k-1\right)
\cos x\sin x\right) } \\
&=&\tfrac{-2x^{2}\cos x+x\sin x+\cos x\sin ^{2}x}{\left( \sin x-x\cos
x\right) \left( x-\cos x\sin x\right) }\times \tfrac{x\sin ^{2}x}{k\left(
\sin x-x\cos x\right) \cos x+\left( x-\cos x\sin x\right) } \\
&:&=D_{1}\left( x\right) \times D_{2}\left( x\right) .
\end{eqnarray*}%
It is known that the function $D_{1}$ (which is equal to $G$ in \cite[Proof
of Lemma 2.9]{Zhu.AAA.2009}) is positive and increasing on $\left( 0,\pi
/2\right) $ proved in \cite[Proof of Lemma 2.9]{Zhu.AAA.2009}, and it
remains to prove the function $D_{2}$ is also positive and increasing.
Clearly, $D_{2}\left( x\right) >0$, we only need to show that $D_{2}^{\prime
}\left( x\right) >0$ for $x\in \left( 0,\pi /2\right) $. Indeed,
differentiation and simplifying yield%
\begin{eqnarray*}
D_{2}^{\prime }\left( x\right) &=&\left( k-1\right) \left( \sin x\right) 
\frac{\left( -2x^{2}\cos x+\cos x\sin ^{2}x+x\sin x\right) }{\left( k\left(
\sin x-x\cos x\right) \cos x+\left( x-\cos x\sin x\right) \right) ^{2}} \\
&=&\frac{\left( k-1\right) x^{2}\sin x\cos x}{\left( k\left( \sin x-x\cos
x\right) \cos x+\left( x-\cos x\sin x\right) \right) ^{2}}\left( \left( 
\frac{\sin x}{x}\right) ^{2}+\frac{\tan x}{x}-2\right) ,
\end{eqnarray*}%
which is clearly positive due to Wilker inequality (\ref{W}). Hence, $%
C/\left( kA+B\right) $ is increasing on $\left( 0,\pi /2\right) $, and it is
deduced that 
\begin{equation*}
\tfrac{5}{12\left( k+2\right) }=\lim_{x\rightarrow 0}\dfrac{C\left( x\right) 
}{kA\left( x\right) +B\left( x\right) }<D\left( x\right) <\lim_{x\rightarrow
\pi /2-}\dfrac{C\left( x\right) }{kA\left( x\right) +B\left( x\right) }=1.
\end{equation*}

This completes the proof.
\end{proof}

\begin{lemma}
\label{Lemma EF/G}Let $U$, $V$ and $W$ be defined on $\left( 0,\infty
\right) $ by%
\begin{eqnarray}
E &=&E\left( x\right) =\left( \cosh x\right) \left( \sinh x-x\cosh x\right)
^{2}\left( x-\cosh x\sinh x\right) ,  \label{E} \\
F &=&F\left( x\right) =\left( \sinh x-x\cosh x\right) \left( x-\cosh x\sinh
x\right) ^{2},  \label{F} \\
G &=&G\left( x\right) =x\left( \sinh ^{2}x\right) \left( 2x^{2}\cosh
x-x\sinh x-\cosh x\sinh ^{2}x\right) .  \label{G}
\end{eqnarray}%
Then for fixed $k\geq 1$ (or $k<-2$) the function $x\mapsto G\left( x\right)
/\left( kE\left( x\right) +F\left( x\right) \right) $ is decreasing
(increasing) on $\left( 0,\infty \right) $. Moreover, we have%
\begin{equation}
\min \left( 0,\dfrac{12}{5\left( k+2\right) }\right) <\dfrac{G(x)}{kE\left(
x\right) +F\left( x\right) }<\max \left( 0,\dfrac{12}{5\left( k+2\right) }%
\right) .  \label{REF/G}
\end{equation}
\end{lemma}

\begin{proof}
It is easy to verify that $E$, $F<0$ for $x\in \left( 0,\infty \right) $ due
to 
\begin{eqnarray*}
\left( x-\cosh x\sinh x\right) &=&\left( 2x-\sinh 2x\right) /2<0, \\
\left( \sinh x-x\cosh x\right) &=&x\left( \frac{\sinh x}{x}-\cos x\right) <0.
\end{eqnarray*}%
While $G<0$ because 
\begin{equation*}
\left( 2x^{2}\cosh x-x\sinh x-\cosh x\sinh ^{2}x\right) =-x^{2}\left( \cosh
x\right) \left( \left( \frac{\sinh x}{x}\right) ^{2}+\frac{\tanh x}{x}%
-2\right) <0
\end{equation*}%
by Wilker inequality (\ref{Zhwh}).

Denote $G/\left( kE+F\right) $ by $H$ and factoring give 
\begin{eqnarray*}
H\left( x\right) &=&\tfrac{x\left( \sinh ^{2}x\right) \left( 2x^{2}\cosh
x-x\sinh x-\cosh x\sinh ^{2}x\right) }{\left( \cosh x\right) \left( \sinh
x-x\cosh x\right) ^{2}\left( x-\sinh x\cosh x\right) k+\left( \sinh x-x\cosh
x\right) \left( x-\sinh x\cosh x\right) ^{2}} \\
&=&\tfrac{-2x^{2}\cosh x+x\sinh x+\cosh x\sinh ^{2}x}{\left( x\cosh x-\sinh
x\right) \left( \sinh x\cosh x-x\right) }\times \tfrac{x\left( \sinh
^{2}x\right) }{\left( k\left( x\cosh x-\sinh x\right) \cosh x+\sinh x\cosh
x-x\right) } \\
&:&=H_{1}\left( x\right) \times H_{2}\left( x\right) .
\end{eqnarray*}%
Clearly, $H_{1}\left( x\right) >0$, and it has shown in \cite[Proof of Lemma
2.2]{Matejicka.IJOPCM.4.1.2011} that $H_{1}$ (that is, the function $s$, in 
\cite[Proof of Lemma 2.2]{Matejicka.IJOPCM.4.1.2011}) is decreasing on $%
\left( 0,\infty \right) $. In order to prove the monotonicity of $H$, we
also need to deal with the sign and monotonicity of $H_{2}$.

(i) Clearly, $H_{2}\left( x\right) >0$ for $k\geq 1$. And, we claim that $%
H_{2}$ is also decreasing on $\left( 0,\infty \right) $. Indeed,
differentiation and simplifying yield%
\begin{eqnarray*}
H_{2}^{\prime }\left( x\right) &=&-\left( k-1\right) \sinh x\frac{\left(
-2x^{2}\cosh x+\cosh x\sinh ^{2}x+x\sinh x\right) }{\left( x\cosh x-\sinh
x\right) ^{2}\left( \cosh x\sinh x-x\right) ^{2}} \\
&=&-\frac{\left( k-1\right) x^{2}\sinh x\cosh x}{\left( x\cosh x-\sinh
x\right) ^{2}\left( \cosh x\sinh x-x\right) ^{2}}\left( \left( \frac{\sinh x%
}{x}\right) ^{2}+\frac{\tanh x}{x}-2\right) <0.
\end{eqnarray*}

Consequently, $H=H_{1}\times H_{2}$ is positive and decreasing on $\left(
0,\infty \right) $, and so%
\begin{equation*}
0=\lim_{x\rightarrow \infty }\dfrac{G(x)}{kE\left( x\right) +F\left(
x\right) }<\dfrac{G(x)}{kE\left( x\right) +F\left( x\right) }%
<\lim_{x\rightarrow 0}\dfrac{G(x)}{kE\left( x\right) +F\left( x\right) }=%
\tfrac{12}{5\left( k+2\right) }.
\end{equation*}%
(ii) For $k<-2$, by the previous proof we see that $-H_{2}^{\prime }$ is
decreasing on $\left( 0,\infty \right) $, and so 
\begin{equation*}
0<-\frac{1}{k}=\lim_{x\rightarrow \infty }\left( -H_{2}\left( x\right)
\right) <-H_{2}\left( x\right) <\lim_{x\rightarrow 0}\left( -H_{2}\left(
x\right) \right) =-\frac{3}{k+2}.
\end{equation*}%
It is implied that $-H_{2}$ is positive and decreasing on $\left( 0,\infty
\right) $, and so is the function $-H=H_{1}\times \left( -H_{2}\right) $.
That is, $H$ is negative and increasing on $\left( 0,\infty \right) $, and (%
\ref{REF/G}) naturally holds.

This completes the proof.
\end{proof}

\begin{remark}
It should be noted that $kE\left( x\right) +F\left( x\right) <0$ for $k\geq
1 $ and $kE\left( x\right) +F\left( x\right) >0$ for $k<-2$. In fact, it
suffices to notice (\ref{REF/G}) and $G(x)<0$ for $x\in \left( 0,\infty
\right) $.
\end{remark}

\begin{lemma}
\label{Lemma 3}For $k\geq 1$, we have%
\begin{equation*}
1>\frac{\ln \left( k+2\right) -\ln 2}{k\left( \ln \pi -\ln 2\right) }>\frac{%
12}{5\left( k+2\right) }.
\end{equation*}
\end{lemma}

\begin{proof}
It suffices to show that 
\begin{eqnarray*}
\delta _{1}\left( k\right) &=&\frac{\ln \left( k+2\right) -\ln 2}{\ln \pi
-\ln 2}-k<0, \\
\delta _{2}\left( k\right) &=&\frac{\ln \left( k+2\right) -\ln 2}{\ln \pi
-\ln 2}-\frac{12k}{5\left( k+2\right) }>0
\end{eqnarray*}%
where $k\geq 1$.

Differentiation gives%
\begin{eqnarray*}
\delta _{1}^{\prime }\left( k\right) &=&\frac{1}{\left( \ln \pi -\ln
2\right) \left( k+2\right) }-1<0, \\
\delta _{2}^{\prime }\left( k\right) &=&\frac{1}{5}\frac{5k+24\ln 2-24\ln
\pi +10}{\left( k+2\right) ^{2}\left( \ln \pi -\ln 2\right) }>0
\end{eqnarray*}%
for $k\geq 1$. It follows that $\delta _{1}\left( k\right) \leq \delta
_{1}\left( 1\right) =\left( \ln 3-\ln 2\right) /\left( \ln 3-\ln \pi \right)
<0$, $\delta \left( k\right) \geq \delta \left( 1\right) =\left( \ln 3-\ln
2\right) /\left( \ln \pi -\ln 2\right) -4/5>0$, which proves the lemma.
\end{proof}

\section{Main results}

\begin{theorem}
\label{Main 1}For fixed $k\geq 1$, the inequality (\ref{Mt}) holds for $x\in
\left( 0,\pi /2\right) $ if and only if $p>0$ or $p\leq -\frac{\ln \left(
k+2\right) -\ln 2}{k\left( \ln \pi -\ln 2\right) }$
\end{theorem}

\begin{proof}
The inequality (\ref{Mt}) is equivalent to%
\begin{equation}
f\left( x\right) =\frac{2}{k+2}\left( \frac{\sin x}{x}\right) ^{kp}+\frac{k}{%
k+2}\left( \frac{\tan x}{x}\right) ^{p}-1>0  \label{f}
\end{equation}%
for $x\in \left( 0,\pi /2\right) $. Differentiation yields%
\begin{eqnarray}
f^{\prime }\left( x\right) &=&-\tfrac{2kp}{k+2}\frac{\sin x-x\cos x}{x^{2}}%
\left( \frac{\sin x}{x}\right) ^{kp-1}+\tfrac{kp}{k+2}\frac{x-\sin x\cos x}{%
x^{2}\cos ^{2}x}\left( \frac{\tan x}{x}\right) ^{p-1}  \notag \\
&=&\tfrac{kp}{k+2}\frac{x-\sin x\cos x}{x^{2}\cos ^{2}x}\left( \frac{\tan x}{%
x}\right) ^{p-1}g(x),  \label{df}
\end{eqnarray}%
where 
\begin{equation}
g(x)=1-4\frac{\sin x-x\cos x}{2x-\sin 2x}\left( \frac{\sin x}{x}\right)
^{\left( k-1\right) p}\left( \cos x\right) ^{p+1}.  \label{g}
\end{equation}%
Simple computation leads to $g(0^{+})=0$.

Differentiation again and simplifying give

\begin{equation}
g^{\prime }\left( x\right) =8\frac{\left( \frac{\sin x}{x}\right) ^{\left(
k-1\right) p}\left( \cos x\right) ^{p}}{x\left( \sin x\right) \left( 2x-\sin
2x\right) ^{2}}h\left( x\right) ,  \label{dg}
\end{equation}%
where 
\begin{eqnarray}
h\left( x\right)  &=&\left( \cos x\right) \left( \sin x-x\cos x\right)
^{2}\left( x-\cos x\sin x\right) kp  \notag \\
&&+\left( x-\cos x\sin x\right) ^{2}\left( \sin x-x\cos x\right) p  \notag \\
&&+x\left( \sin ^{2}x\right) \left( -2x^{2}\cos x+x\sin x+\cos x\sin
^{2}x\right)   \notag \\
&=&kpA\left( x\right) +pB\left( x\right) +C\left( x\right)   \label{h} \\
&=&\left( kA+B\right) \left( p+\frac{C}{kA+B}\right) ,  \notag
\end{eqnarray}%
here $A\left( x\right) ,B\left( x\right) ,C\left( x\right) $ are defined by (%
\ref{A}), (\ref{B}), (\ref{C}), respectively.

By (\ref{df}), (\ref{dg}) we easily get 
\begin{eqnarray}
\limfunc{sgn}f^{\prime }\left( x\right) &=&\limfunc{sgn}p\limfunc{sgn}g(x),
\label{sgn(df)} \\
\limfunc{sgn}g^{\prime }\left( x\right) &=&\limfunc{sgn}h\left( x\right) .
\label{sgn(dg)}
\end{eqnarray}

\textbf{Necessity}. We first present two limit relations:%
\begin{eqnarray}
\lim_{x\rightarrow 0^{+}}x^{4}f\left( x\right) &=&\frac{kp}{36}\left( p+%
\frac{12}{5\left( k+2\right) }\right) ,  \label{Limit1} \\
\lim_{x\rightarrow \left( \pi /2\right) ^{-}}f\left( x\right) &=&\left\{ 
\begin{array}{cc}
\infty & \text{if }p>0, \\ 
\frac{2}{k+2}\left( \frac{2}{\pi }\right) ^{kp}-1 & \text{if }p<0.%
\end{array}%
\right.  \label{Limit2}
\end{eqnarray}%
In fact, using power series extension yields%
\begin{equation*}
f\left( x\right) =\frac{kp}{36}\frac{kp+2p+12/5}{k+2}x^{4}+O\left(
x^{6}\right) ,
\end{equation*}%
which implies the first limit relation (\ref{Limit1}). From the fact $%
\lim_{x\rightarrow \pi /2^{-}}\tan x=\infty $ the second one (\ref{Limit2})
easily follows.

Now we can derive the necessary condition for (\ref{Mt}) to holds for $x\in
\left( 0,\pi /2\right) $ from the simultaneous inequalities $%
\lim_{x\rightarrow 0^{+}}x^{4}f\left( x\right) \geq 0$ and $%
\lim_{x\rightarrow \left( \pi /2\right) ^{-}}f\left( x\right) \geq 0$.
Solving for $p$ yields $p>0$ or 
\begin{equation*}
p\leq \min \left( -\frac{12}{5\left( k+2\right) },-\frac{\ln \left(
k+2\right) -\ln 2}{k\left( \ln \pi -\ln 2\right) }\right) =-\frac{\ln \left(
k+2\right) -\ln 2}{k\left( \ln \pi -\ln 2\right) },
\end{equation*}%
where the equality holds is due to the Lemma \ref{Lemma 3}.

\textbf{Sufficiency}. We prove the condition $p>0$ or $p\leq -\frac{\ln
\left( k+2\right) -\ln 2}{k\left( \ln \pi -\ln 2\right) }$ is sufficient. We
distinguish three cases.

Case 1: $p>0$. Clearly, $h\left( x\right) >0$, then $g^{\prime }\left(
x\right) >0$, and then $g\left( x\right) >g\left( 0^{+}\right) =0$, which
together with $\func{sgn}p=1$ yields $f^{\prime }\left( x\right) >0$. Then $%
f\left( x\right) >f\left( 0^{+}\right) =0$.

Case 2: $p\leq -1$. By Lemma (\ref{Lemma AB/C}) it is easy to get 
\begin{equation*}
p+\frac{C}{kA+B}<p+1\leq 0,
\end{equation*}%
which reveals that $h\left( x\right) <0$, then $g^{\prime }\left( x\right)
<0 $, and then $g\left( x\right) <g\left( 0^{+}\right) =0$, which in
combination with $\func{sgn}p=-1$ implies $f^{\prime }\left( x\right) >0$.
Then $f\left( x\right) >f\left( 0^{+}\right) =0$.

Case 3: $-1<p\leq -\frac{\ln \left( k+2\right) -\ln 2}{k\left( \ln \pi -\ln
2\right) }$. Lemma (\ref{Lemma AB/C}) reveals that $\frac{C}{kA+B}$ is
increasing on $\left( 0,\pi /2\right) $, so is the function $x\mapsto p+%
\frac{C}{kA+B}:=\lambda \left( x\right) $. Since 
\begin{equation*}
\lambda \left( 0^{+}\right) =p+\frac{12}{5\left( k+2\right) }<0\text{, \ }%
\lambda \left( \frac{\pi }{2}^{-}\right) =p+1>0,
\end{equation*}%
there is a unique $x_{1}\in \left( 0,\pi /2\right) $ such that $\lambda
\left( x\right) <0$ for $x\in \left( 0,x_{1}\right) $ and $\lambda \left(
x\right) >0$ for $x\in \left( x_{1},\pi /2\right) $, and so is $g^{\prime
}\left( x\right) $. Therefore, $g\left( x\right) <g\left( 0^{+}\right) =0$
for $x\in \left( 0,x_{1}\right) $ but $g\left( \pi /2^{-}\right) =1$, which
implies that there is a sole $x_{0}\in \left( x_{1},\pi /2\right) $ such
that $g\left( x\right) <0$ for $x\in \left( 0,x_{0}\right) $ and $g\left(
x\right) >0$ for $x\in \left( x_{0},\pi /2\right) $. Due to $\func{sgn}p=-1$
it is deduced that $f^{\prime }\left( x\right) >0$ for $x\in \left(
0,x_{0}\right) $ and $f^{\prime }\left( x\right) <0$ for $x\in \left(
x_{0},\pi /2\right) $, which reveals that $f$ is increasing on $\left(
0,x_{0}\right) $ and decreasing on $\left( x_{0},\pi /2\right) $. It follows
that 
\begin{eqnarray*}
0 &=&f\left( 0^{+}\right) <f\left( x\right) <f\left( x_{0}\right) =0\text{
for }x\in \left( 0,x_{0}\right) , \\
f\left( x_{0}\right) &>&f\left( x\right) >f\left( \pi /2^{-}\right) =\tfrac{2%
}{k+2}\left( \tfrac{2}{\pi }\right) ^{kp}-1\geq 0\text{ for }x\in \left(
x_{0},\pi /2\right) ,
\end{eqnarray*}%
that is, $f\left( x\right) >0$ for $x\in \left( 0,\pi /2\right) $.

This completes the proof.
\end{proof}

\begin{theorem}
\label{Main 2}For fixed $k\geq 1$, the reverse of (\ref{Mt}), that is, 
\begin{equation}
\frac{2}{k+2}\left( \frac{\sin x}{x}\right) ^{kp}+\frac{k}{k+2}\left( \frac{%
\tan x}{x}\right) ^{p}<1  \label{Mtr}
\end{equation}%
holds for $x\in \left( 0,\pi /2\right) $ if and only if $-\frac{12}{5\left(
k+2\right) }\leq p<0$.
\end{theorem}

\begin{proof}
\textbf{Necessity}. If inequality (\ref{Mtr}) holds for $x\in \left( 0,\pi
/2\right) $, then we have 
\begin{equation*}
\lim_{x\rightarrow 0^{+}}\frac{f\left( x\right) }{x^{4}}=\frac{kp}{36}\left(
p+\frac{12}{5\left( k+2\right) }\right) \leq 0.
\end{equation*}

Solving the inequalities for $p$ yields $-\frac{12}{5\left( k+2\right) }\leq
p<0$.

\textbf{Sufficiency}. We prove the condition $-\frac{12}{5\left( k+2\right) }%
\leq p<0$ is sufficient. It suffices to show that $f\left( x\right) <0$ for $%
x\in \left( 0,\pi /2\right) $. By Lemma (\ref{Lemma AB/C}) it is easy to get 
\begin{equation*}
p+\frac{C}{kA+B}\geq p+\frac{12}{5\left( k+2\right) }\geq 0,
\end{equation*}%
which reveals that $h\left( x\right) >0$, then $g^{\prime }\left( x\right)
>0 $, and then $g\left( x\right) >g\left( 0^{+}\right) =0$. It in
combination with $\func{sgn}p=-1$ implies $f^{\prime }\left( x\right) <0$.
Thus, $f\left( x\right) <f\left( 0^{+}\right) =0$, which proves the
sufficiency and the proof is complete.
\end{proof}

\begin{theorem}
\label{Main 3}For fixed $k\geq 1$, the inequality (\ref{Mh}) holds for $x\in
\left( 0,\infty \right) $ if and only if $p>0$ or $p\leq -\frac{12}{5\left(
k+2\right) }$.
\end{theorem}

\begin{proof}
We define 
\begin{equation}
u\left( x\right) =\frac{2}{k+2}\left( \frac{\sinh x}{x}\right) ^{kp}+\frac{k%
}{k+2}\left( \frac{\tanh x}{x}\right) ^{p}-1.  \label{u}
\end{equation}%
Then inequality (\ref{Mh}) is equivalent to $u\left( x\right) >0$.
Differentiation leads to%
\begin{equation}
u^{\prime }\left( x\right) =-\frac{kp}{2\left( k+2\right) }\frac{\sinh 2x-2x%
}{x^{2}\cosh ^{2}x}\left( \frac{\tanh x}{x}\right) ^{p-1}v\left( x\right) ,
\label{du}
\end{equation}%
where 
\begin{equation}
v\left( x\right) =1-4\frac{\sinh x-x\cosh x}{2x-\sinh 2x}\left( \frac{\sinh x%
}{x}\right) ^{kp-p}\left( \cosh x\right) ^{p+1}.  \label{v}
\end{equation}%
Differentiation again gives%
\begin{equation}
v^{\prime }\left( x\right) =\frac{2\left( \cosh ^{p}x\right) \left( \frac{%
\sinh x}{x}\right) ^{kp-p}}{\left( x\sinh x\right) \left( x-\cosh x\sinh
x\right) ^{2}}w\left( x\right) ,  \label{dv}
\end{equation}%
where 
\begin{eqnarray}
w\left( x\right) &=&\left( \cosh x\right) \left( \sinh x-x\cosh x\right)
^{2}\left( x-\cosh x\sinh x\right) kp  \notag \\
&&+\left( \sinh x-x\cosh x\right) \left( x-\cosh x\sinh x\right) ^{2}p 
\notag \\
&&+x\left( \sinh ^{2}x\right) \left( 2x^{2}\cosh x-x\sinh x-\cosh x\sinh
^{2}x\right)  \notag \\
&=&kpE\left( x\right) +pF\left( x\right) +G\left( x\right) =\left(
kE+F\right) \left( p+\frac{G}{kE+F}\right) ,  \label{w}
\end{eqnarray}%
here $E\left( x\right) ,F\left( x\right) ,G\left( x\right) $ are defined by (%
\ref{E}), (\ref{F}), (\ref{G}), respectively.

By (\ref{du}), (\ref{dv}) we easily get 
\begin{eqnarray}
\limfunc{sgn}u^{\prime }\left( x\right) &=&-\limfunc{sgn}\frac{k}{k+2}%
\limfunc{sgn}p\limfunc{sgn}v(x),  \label{sgn(du)} \\
\limfunc{sgn}v^{\prime }\left( x\right) &=&\limfunc{sgn}w\left( x\right) .
\label{sgn(dv)}
\end{eqnarray}

\textbf{Necessity}. If inequality (\ref{Mh}) holds for $x\in \left( 0,\infty
\right) $, then we have $\lim_{x\rightarrow 0^{+}}x^{-4}u\left( x\right)
\geq 0$. Expanding $u\left( x\right) $ in power series gives 
\begin{equation*}
u\left( x\right) =\frac{k}{36}p\left( p+\frac{12}{5p\left( k+2\right) }%
\right) x^{4}+O\left( x^{6}\right) .
\end{equation*}%
Hence we get 
\begin{equation*}
\lim_{x\rightarrow 0^{+}}x^{-4}u\left( x\right) =\frac{k}{36}p\left( p+\frac{%
12}{5\left( k+2\right) }\right) \geq 0.
\end{equation*}%
Solving the inequality for $p$ yields $p>0$ or $p\leq -\frac{12}{5\left(
k+2\right) }$.

\textbf{Sufficiency}. We prove the condition $p>0$ or $p\leq -\frac{12}{%
5\left( k+2\right) }$ is sufficient for (\ref{Mh}) to hold.

If $p>0$, then $w\left( x\right) <0$ due to $E,F,G<0$. Hence, from (\ref%
{sgn(dv)}) we have $v^{\prime }\left( x\right) <0$, and then $v\left(
x\right) <\lim_{x\rightarrow 0^{+}}v\left( x\right) =0$. It is derived by (%
\ref{sgn(du)}) that $u^{\prime }\left( x\right) >0$, and so $u\left(
x\right) >\lim_{x\rightarrow 0^{+}}u\left( x\right) =0$.

If $p\leq -\frac{12}{5\left( k+2\right) }$, then by Lemma \ref{Lemma EF/G}
we have%
\begin{equation*}
p+\frac{G}{kE+F}\leq -\frac{12}{5\left( k+2\right) }+\frac{G}{kE+F}<0,
\end{equation*}%
and then%
\begin{equation*}
w\left( x\right) =\left( kE+F\right) \left( p+\frac{G}{kE+F}\right) >0.
\end{equation*}%
From (\ref{sgn(dv)}) we have $v^{\prime }\left( x\right) >0$, and then $%
v\left( x\right) >\lim_{x\rightarrow 0^{+}}v\left( x\right) =0$. It follows
by (\ref{sgn(du)}) that $u^{\prime }\left( x\right) >0$, which implies that $%
u\left( x\right) >\lim_{x\rightarrow 0^{+}}u\left( x\right) =0$.

This completes the proof.
\end{proof}

\begin{remark}
For $k\geq 1$, since $\lim_{x\rightarrow \infty }u\left( x\right) =\infty $
for $p\neq 0$ and $\lim_{x\rightarrow \infty }u\left( x\right) =0$ for $p=0$%
, there has no $p$ such that the reverse inequality of (\ref{Mh}) holds for
all $x>0$. But we can show that there is a unique $x_{0}\in \left( 0,\infty
\right) $ such that $u\left( x\right) <0$, that is, the reverse inequality
of (\ref{Mh}), for $-\frac{12}{5\left( k+2\right) }<p<0$. The details of
proof are omitted.
\end{remark}

\begin{theorem}
\label{Main 4}For fixed $k<-2$, the reverse of (\ref{Mh}), that is,%
\begin{equation}
\frac{2}{k+2}\left( \frac{\sinh x}{x}\right) ^{kp}+\frac{k}{k+2}\left( \frac{%
\tanh x}{x}\right) ^{p}<1  \label{Mhr}
\end{equation}
holds for $x\in \left( 0,\infty \right) $ if and only if $p<0$ or $p\geq -%
\frac{12}{5\left( k+2\right) }$.
\end{theorem}

\begin{proof}
\textbf{Necessity}. If inequality (\ref{Mh}) holds for $x\in \left( 0,\infty
\right) $, then we have 
\begin{equation*}
\lim_{x\rightarrow 0^{+}}\frac{u\left( x\right) }{x^{4}}=\frac{k}{36}p\left(
p+\frac{12}{5\left( k+2\right) }\right) \leq 0.
\end{equation*}%
Solving the inequality for $p$ yields $p<0$ or $p\geq -\frac{12}{5\left(
k+2\right) }$.

\textbf{Sufficiency}. We prove the condition $p<0$ or $p\geq -\frac{12}{%
5\left( k+2\right) }$ is sufficient for (\ref{Mh}) to hold.

If $p<0$, then $w\left( x\right) =\left( kE+F\right) \left( p+\frac{G}{kE+F}%
\right) <0$ due to $kE+F>0$ and $G<0$. Hence, from (\ref{sgn(dv)}) we have $%
v^{\prime }\left( x\right) <0$, and then $v\left( x\right)
<\lim_{x\rightarrow 0^{+}}v\left( x\right) =0$. It is derived by (\ref%
{sgn(du)}) that $u^{\prime }\left( x\right) <0$, and so $u\left( x\right)
<\lim_{x\rightarrow 0^{+}}u\left( x\right) =0$.

If $p\geq -\frac{12}{5\left( k+2\right) }$, then by Lemma \ref{Lemma EF/G}
we have%
\begin{equation*}
p+\frac{G}{kE+F}\geq p+\frac{12}{5\left( k+2\right) }>0,
\end{equation*}%
and then%
\begin{equation*}
w\left( x\right) =\left( kE+F\right) \left( p+\frac{G}{kE+F}\right) >0.
\end{equation*}%
From (\ref{sgn(dv)}) we have $v^{\prime }\left( x\right) >0$, and then $%
v\left( x\right) >\lim_{x\rightarrow 0^{+}}v\left( x\right) =0$. It follows
by (\ref{sgn(du)}) that $u^{\prime }\left( x\right) <0$, which implies that $%
u\left( x\right) <\lim_{x\rightarrow 0^{+}}u\left( x\right) =0$.

This completes the proof.
\end{proof}

\section{Applications}

\subsection{Huygens type inequalities}

Letting $k=1$ in Theorem \ref{Main 1} and \ref{Main 2}, we have

\begin{proposition}
\label{Ptk=1}For $x\in \left( 0,\pi /2\right) $, inequality 
\begin{equation}
\frac{2}{3}\left( \frac{\sin x}{x}\right) ^{p}+\frac{1}{3}\left( \frac{\tan x%
}{x}\right) ^{p}>1>\frac{2}{3}\left( \frac{\sin x}{x}\right) ^{q}+\frac{1}{3}%
\left( \frac{\tan x}{x}\right) ^{q}  \label{Pt1}
\end{equation}%
holds if and only if $p>0$ or $p\leq -\frac{\ln 3-\ln 2}{\ln \pi -\ln 2}%
\approx -0.898$ and $-4/5\leq q<0$.
\end{proposition}

Let $M_{r}\left( a,b;w\right) $ denote the $r$-th weighted power mean of
positive numbers $a,b>0$ defined by 
\begin{equation}
M_{r}\left( a,b;w\right) :=\left( wa^{r}+\left( 1-w\right) b^{r}\right)
^{1/r}\text{ if }r\neq 0\text{ and }M_{0}\left( a,b;w\right) =a^{w}b^{1-w},
\label{M_r}
\end{equation}%
where $w\in \left( 0,1\right) $.

Since%
\begin{equation*}
\frac{2}{3}\left( \frac{\sin x}{x}\right) ^{p}+\frac{1}{3}\left( \frac{\tan x%
}{x}\right) ^{p}=\frac{\frac{2}{3}+\frac{1}{3}\left( \cos x\right) ^{-p}}{%
\left( \frac{\sin x}{x}\right) ^{-p}},
\end{equation*}%
by Proposition \ref{Ptk=1} the inequality 
\begin{equation*}
\frac{\sin x}{x}>\left( \frac{2}{3}+\frac{1}{3}\left( \cos x\right)
^{-p}\right) ^{-1/p}=M_{-p}\left( 1,\cos x;\tfrac{2}{3}\right) 
\end{equation*}%
holds for $x\in \left( 0,\pi /2\right) $ if and only if $-p\leq 4/5$.
Similarly, its reverse one holds if and only if $-p\geq \frac{\ln 3-\ln 2}{%
\ln \pi -\ln 2}$. The facts cab be stated as a corollary.

\begin{corollary}
\label{Ctyang1}Let $M_{r}\left( a,b;w\right) $ be defined by (\ref{M_r}).
Then for $x\in \left( 0,\pi /2\right) $, the inequalities 
\begin{equation}
M_{\alpha }\left( 1,\cos x;\tfrac{2}{3}\right) <\frac{\sin x}{x}<M_{\beta
}\left( 1,\cos x;\tfrac{2}{3}\right)   \label{Yang1t}
\end{equation}%
hold if and only if $\alpha \leq 4/5$ and $\beta \geq \frac{\ln 3-\ln 2}{\ln
\pi -\ln 2}\approx -0.898$.
\end{corollary}

\begin{remark}
Cusa-Huygens inequality \cite{Huygens.1888-1940} refers to%
\begin{equation}
\frac{\sin x}{x}<\tfrac{2}{3}+\tfrac{1}{3}\cos x  \label{Cusa}
\end{equation}%
holds for $x\in \left( 0,\pi /2\right) $, which is an equivalent one of the
second one in (\ref{Nueman}). As an improvement and generalization,
Corollary \ref{Ctyang1} was proved in \cite{Yang.MIA.2013.inprint} by Yang.
Here we provide a new proof.
\end{remark}

\begin{remark}
Let $a>b>0$ and let $x=\arcsin \frac{a-b}{a+b}\in \left( 0,\pi /2\right) $.
Then $(\sin x)/x=P/A$, $\cos x=G/A$, and then inequalities (\ref{Yang1t})
can be changed into 
\begin{equation}
M_{\alpha }\left( A,G;\tfrac{2}{3}\right) <P<M_{\beta }\left( A,G;\tfrac{2}{3%
}\right) ,  \label{P-A-G}
\end{equation}%
where $P$ is the first Seiffert mean \cite{Seiffert.EM.42.1987} defined by 
\begin{equation*}
P=P\left( a,b\right) =\frac{a-b}{2\arcsin \frac{a-b}{a+b}},
\end{equation*}%
$A$ and $G$ denote the arithmetic and geometric means of $a$ and $b$,
respectively.

Let $x=\arctan \frac{a-b}{a+b}$. Then $(\sin x)/x=T/Q$, $\cos x=A/Q$, and
then inequalities (\ref{Yang1t}) can be changed into 
\begin{equation}
M_{\alpha }\left( Q,A;\tfrac{2}{3}\right) <T<M_{\beta }\left( Q,A;\tfrac{2}{3%
}\right) ,  \label{T-A-Q}
\end{equation}%
where $T$ is the second Seiffert mean \cite{Seiffert.DW.29.1995} defined by 
\begin{equation*}
T=T\left( a,b\right) =\frac{a-b}{2\arctan \frac{a-b}{a+b}},
\end{equation*}%
$Q$ denotes the quadratic mean of $a$ and $b$.

Obviously, by Corollary \ref{Chyang1}, both the two double inequalities (\ref%
{P-A-G}) (see \cite{Yang.MIA.2013.inprint}) and (\ref{T-A-Q}) hold if and
only if $\alpha \leq 4/5$ and $\beta \geq \frac{\ln 3-\ln 2}{\ln \pi -\ln 2}%
\approx -0.898$, in which (\ref{T-A-Q}) seem to be new ones
\end{remark}

In the same way, taking $k=1$ in Theorem \ref{Main 3}

\begin{proposition}
\label{Phk=1}For $x\in \left( 0,\infty \right) $, inequality 
\begin{equation}
\frac{2}{3}\left( \frac{\sinh x}{x}\right) ^{p}+\frac{1}{3}\left( \frac{%
\tanh x}{x}\right) ^{p}>1  \label{Ph1}
\end{equation}%
holds if and only if $p>0$ or $p\leq -\frac{4}{5}$.
\end{proposition}

Similar to Corollary \ref{Ctyang1}, we have

\begin{corollary}
\label{Chyang1}Let $M_{r}\left( a,b;w\right) $ be defined by (\ref{M_r}).
Then for $x\in \left( 0,\infty \right) $, the inequalities 
\begin{equation}
M_{\alpha }\left( 1,\cosh x;\tfrac{2}{3}\right) <\frac{\sinh x}{x}<M_{\beta
}\left( 1,\cosh x;\tfrac{2}{3}\right)   \label{Yang1h}
\end{equation}%
hold if and only if $\alpha \leq 0$ and $\beta \geq 4/5$.
\end{corollary}

\begin{remark}
Let $a>b>0$ and let $x=\ln \sqrt{a/b}$. Then $\left( \sinh x\right) /x=L/G$, 
$\cosh x=A/G$, and then (\ref{Yang1h}) can be changed into 
\begin{equation}
M_{\alpha }\left( G,A;\tfrac{2}{3}\right) <L<M_{\beta }\left( G,A;\tfrac{2}{3%
}\right) ,  \label{L-A-G}
\end{equation}%
where $L$ is the logarithmic means of $a$ and $b$ defined by%
\begin{equation*}
L=L\left( a,b\right) =\frac{a-b}{\ln a-\ln b}.
\end{equation*}%
Making a change of variable $x=\func{arcsinh}\frac{b-a}{a+b}$ yields $(\sinh
x)/x=NS/A$, $\cosh x=Q/A$, where $NS$ is the Nueman-S\'{a}ndor mean defined
by%
\begin{equation*}
NS=NS\left( a,b\right) =\frac{a-b}{2\func{arcsinh}\frac{a-b}{a+b}}.
\end{equation*}%
Thus, (\ref{Yang1h}) is equivalent to%
\begin{equation}
M_{\alpha }\left( A,Q;\tfrac{2}{3}\right) <NS<M_{\beta }\left( A,Q;\tfrac{2}{%
3}\right) .  \label{NS-A-Q}
\end{equation}

Corollary \ref{Chyang1} implies that the inequalities (\ref{L-A-G}) and (\ref%
{NS-A-Q}) hold if and only if $\alpha \leq 0$ and $\beta \geq 4/5$. The
second one in (\ref{NS-A-Q}) is a new one.
\end{remark}

\begin{remark}
It should be pointed out that all inequalities involving $(\sin x)/x\ $and $%
\cos x$ or $(\sinh x)/x\ $and $\cosh x$ in this paper can be changed into
the equivalent ones for means by variable substitutions mentioned
previously. In what follows we no longer mention.
\end{remark}

\subsection{Wilker-Zhu type inequalities}

Letting $k=2$ in Theorem \ref{Main 1} and \ref{Main 2}, we have

\begin{proposition}
\label{Ptk=2}For $x\in \left( 0,\pi /2\right) $, inequality 
\begin{equation}
\left( \frac{\sin x}{x}\right) ^{2p}+\left( \frac{\tan x}{x}\right)
^{p}>2>\left( \frac{\sin x}{x}\right) ^{2q}+\left( \frac{\tan x}{x}\right)
^{q}  \label{Pt2}
\end{equation}%
holds if and only if $p>0$ or $p\leq -\frac{\ln 2}{2\left( \ln \pi -\ln
2\right) }\approx -0.767$ and $-3/5\leq q<0$.
\end{proposition}

Note that 
\begin{equation*}
\frac{\left( \frac{\sin x}{x}\right) ^{2p}+\left( \frac{\tan x}{x}\right)
^{p}-2}{\left( \frac{\sin x}{x}\right) ^{p}+\frac{\sqrt{8+\cos ^{-2p}x}+\cos
^{-p}x}{2}}=\left( \frac{x}{\sin x}\right) ^{-p}-\frac{\sqrt{8+\cos ^{-2p}x}%
-\cos ^{-p}x}{2},
\end{equation*}%
by Proposition \ref{Ptk=2} the inequality%
\begin{equation*}
\frac{x}{\sin x}>\left( \frac{\sqrt{8+\cos ^{-2p}x}-\cos ^{-p}x}{2}\right)
^{-1/p}
\end{equation*}%
or 
\begin{equation*}
\frac{\sin x}{x}<\left( \frac{\sqrt{8+\cos ^{-2p}x}+\cos ^{-p}x}{4}\right)
^{-1/p}:=H_{-p}\left( \cos x\right)
\end{equation*}%
holds for $x\in \left( 0,\pi /2\right) $ if and only if $-p\geq \frac{\ln 2}{%
2\left( \ln \pi -\ln 2\right) }$, where $H_{r}$ is defined on $\left(
0,\infty \right) $ by 
\begin{equation}
H_{r}\left( t\right) =\left( \frac{\sqrt{8+t^{2r}}+t^{r}}{4}\right) ^{1/r}%
\text{ if }r\neq 0\text{ and }H_{0}\left( t\right) =\sqrt[3]{t}\text{.}
\label{H_r}
\end{equation}%
Likewise, its reverse one holds if and only if $-p\leq 3/5$. This result cab
be stated as a corollary.

\begin{corollary}
\label{Ctyang2}Let $H_{r}\left( t\right) $ be defined by (\ref{H_r}). Then
for $x\in \left( 0,\pi /2\right) $, the inequalities 
\begin{equation}
H_{\alpha }\left( \cos x\right) <\frac{\sin x}{x}<H_{\beta }\left( \cos
x\right)  \label{Yang2t}
\end{equation}%
are true if and only if $\alpha \leq 3/5$ and $\beta \geq \frac{\ln 2}{%
2\left( \ln \pi -\ln 2\right) }\approx 0.767$.
\end{corollary}

Taking $k=2$ in Theorem \ref{Main 3}, we have

\begin{proposition}
\label{Phk=2}For $x\in \left( 0,\infty \right) $, the inequality 
\begin{equation*}
\left( \frac{\sinh x}{x}\right) ^{2p}+\left( \frac{\tanh x}{x}\right) ^{p}>2
\end{equation*}%
holds if and only if $p>0$ or $p\leq -3/5$.
\end{proposition}

In a similar way, we get

\begin{corollary}
\label{Chyang2}Let $H_{r}\left( t\right) $ be defined by (\ref{H_r}). Then
for $x\in \left( 0,\infty \right) $, the inequalities 
\begin{equation}
H_{\alpha }\left( \cosh x\right) <\frac{\sinh x}{x}<H_{\beta }\left( \cosh
x\right)  \label{Yang2h}
\end{equation}%
are true if and only if $\alpha \leq 0$ and $\beta \geq 3/5$.
\end{corollary}

Now we give a generalization of inequalities (\ref{Zhwt}) given by Zhu \cite%
{Zhu.CMA.58.2009}

\begin{proposition}
\label{PtZhug}For fixed $k\geq 1$, both the chains of inequalities 
\begin{eqnarray}
\tfrac{2}{k+2}\left( \tfrac{\sin x}{x}\right) ^{kp}+\tfrac{k}{k+2}\left( 
\tfrac{\tan x}{x}\right) ^{p} &\geq &\tfrac{k}{k+2}\left( \tfrac{\sin x}{x}%
\right) ^{kp}+\tfrac{2}{k+2}\left( \tfrac{\tan x}{x}\right) ^{p}
\label{Yang3t} \\
&>&\tfrac{2}{k+2}\left( \tfrac{x}{\sin x}\right) ^{kp}+\tfrac{k}{k+2}\left( 
\tfrac{x}{\tan x}\right) ^{p}>1,  \notag \\
\tfrac{2}{k+2}\left( \tfrac{\sin x}{x}\right) ^{kp}+\tfrac{k}{k+2}\left( 
\tfrac{\tan x}{x}\right) ^{p} &>&\tfrac{2}{k+2}\left( \tfrac{x}{\tan x}%
\right) ^{p}+\tfrac{k}{k+2}\left( \tfrac{x}{\sin x}\right) ^{kp}
\label{Yang4t} \\
&\geq &\tfrac{2}{k+2}\left( \tfrac{x}{\sin x}\right) ^{kp}+\tfrac{k}{k+2}%
\left( \tfrac{x}{\tan x}\right) ^{p}>1  \notag
\end{eqnarray}%
hold for $x\in \left( 0,\pi /2\right) $ if and only if $k\geq 2$ and $p\geq 
\frac{\ln \left( k+2\right) -\ln 2}{k\left( \ln \pi -\ln 2\right) }$.
\end{proposition}

\begin{proof}
The first inequality in (\ref{Yang3t}) is equivalent to 
\begin{eqnarray*}
&&\frac{2}{k+2}\left( \frac{\sin x}{x}\right) ^{kp}+\frac{k}{k+2}\left( 
\frac{\tan x}{x}\right) ^{p}-\frac{k}{k+2}\left( \frac{\sin x}{x}\right)
^{kp}-\frac{2}{k+2}\left( \frac{\tan x}{x}\right) ^{p} \\
&=&\frac{k-2}{k+2}\left( \left( \frac{\tan x}{x}\right) ^{p}-\left( \frac{%
\sin x}{x}\right) ^{kp}\right) >0.
\end{eqnarray*}%
Due to $\frac{\tan x}{x}>1$ and $\frac{\sin x}{x}<1$, it holds for $x\in
\left( 0,\pi /2\right) $ if and only if 
\begin{equation*}
\left( k,p\right) \in \{k\geq 2,p>0\}\cup \{1\leq k\leq 2,p<0\}:=\Omega _{1}.
\end{equation*}

The second one is equivalent to 
\begin{equation*}
\frac{\frac{k}{k+2}\left( \frac{\sin x}{x}\right) ^{kp}+\frac{2}{k+2}\left( 
\frac{\tan x}{x}\right) ^{p}}{\frac{2}{k+2}\left( \frac{x}{\sin x}\right)
^{kp}+\frac{k}{k+2}\left( \frac{x}{\tan x}\right) ^{p}}>1,
\end{equation*}
which can be simplified to%
\begin{equation*}
\left( \frac{\sin x}{x}\right) ^{kp}\left( \frac{\tan x}{x}\right)
^{p}=\left( \left( \frac{\sin x}{x}\right) ^{k+1}\frac{1}{\cos x}\right)
^{p}>1.
\end{equation*}%
It is true for $x\in \left( 0,\pi /2\right) $ if and only if $\left(
k,p\right) \in \{k+1\geq 3,p\geq 0\}:=\Omega _{2}$.

By Theorem \ref{Main 1}, the third one in (\ref{Yang3t}) holds for $x\in
\left( 0,\pi /2\right) $ if and only if 
\begin{equation*}
\left( k,p\right) \in \{k\geq 1,-p>0\}\cup \{k\geq 1,-p\leq -\frac{\ln
\left( k+2\right) -\ln 2}{k\left( \ln \pi -\ln 2\right) }\}:=\Omega _{3}.
\end{equation*}

Hence, inequalities (\ref{Yang3t}) hold for $x\in \left( 0,\pi /2\right) $
if and only if 
\begin{equation*}
\left( k,p\right) \in \Omega _{1}\cap \Omega _{2}\cap \Omega _{3}=\{k\geq
2,p\geq \frac{\ln \left( k+2\right) -\ln 2}{k\left( \ln \pi -\ln 2\right) }%
\},
\end{equation*}
which proves (\ref{Yang3t}).

In the same way, we can prove (\ref{Yang4t}), of which details are omitted.
\end{proof}

Letting $k=2$ in Proposition \ref{PtZhug} we have

\begin{corollary}
\label{Ctyang3}For $x\in \left( 0,\pi /2\right) $, the inequalities (\ref%
{Zhwt}) hold if and only if $p\geq \frac{\ln 2}{2\left( \ln \pi -\ln
2\right) }\approx 0.767$.
\end{corollary}

Similarly, using Theorem \ref{Main 3} we easily prove the following

\begin{proposition}
\label{PhZhug}For fixed $k\geq 1$, the inequalities 
\begin{equation}
\tfrac{k}{k+2}\left( \tfrac{\sinh x}{x}\right) ^{kp}+\tfrac{2}{k+2}\left( 
\tfrac{\tanh x}{x}\right) ^{p}>\tfrac{2}{k+2}\left( \tfrac{x}{\sinh x}%
\right) ^{kp}+\tfrac{k}{k+2}\left( \tfrac{x}{\tanh x}\right) ^{p}>1
\label{Yang3h}
\end{equation}%
hold $x\in \left( 0,\infty \right) $ if and only if $k\geq 2$ and $p\geq 
\frac{12}{5\left( k+2\right) }$.
\end{proposition}

Letting $k=2$ in Proposition \ref{PhZhug} we have

\begin{corollary}
\label{Chyang3}For $x\in \left( 0,\infty \right) $, the inequalities (\ref%
{Zhwh}) hold if and only if $p\geq 3/5$.
\end{corollary}

\begin{remark}
Clearly, Corollaries \ref{Ctyang3} and \ref{Chyang3} offer another method
for solving the problems posed by Zhu in \cite{Zhu.AAA.2009}.
\end{remark}

\subsection{Other Wilker type inequalities}

Taking $k=3,4$ in Theorems \ref{Main 1} and \ref{Main 2}, we obtain the
following

\begin{proposition}
\label{Ptk=3}For $x\in \left( 0,\pi /2\right) $, inequality 
\begin{equation}
\frac{2}{5}\left( \frac{\sin x}{x}\right) ^{3p}+\frac{3}{5}\left( \frac{\tan
x}{x}\right) ^{p}>1  \label{Pt3}
\end{equation}%
holds if and only if $p>0$ or $p\leq -\frac{\ln 5-\ln 2}{3\left( \ln \pi
-\ln 2\right) }\approx -0.676$. It is reversed if and only if $-12/25\leq
p<0 $.
\end{proposition}

\begin{proposition}
\label{Ptk=4}For $x\in \left( 0,\pi /2\right) $, inequality 
\begin{equation}
\frac{1}{3}\left( \frac{\sin x}{x}\right) ^{4p}+\frac{2}{3}\left( \frac{\tan
x}{x}\right) ^{p}>1  \label{Pt4}
\end{equation}%
holds if and only if $p>0$ or $p\leq -\frac{\ln 3}{4\left( \ln \pi -\ln
2\right) }\approx -0.608$. It is reversed if and only if $-2/5\leq p<0$.
\end{proposition}

Putting $k=-3,-4$ in Theorem \ref{Main 3} we get

\begin{proposition}
\label{Phk=-3}For $x\in \left( 0,\infty \right) $, inequality 
\begin{equation}
\left( \frac{\tanh x}{x}\right) ^{p}<\frac{2}{3}\left( \frac{x}{\sinh x}%
\right) ^{3p}+\frac{1}{3}  \label{Ph-3}
\end{equation}%
holds if and only if $p<0$ or $p\geq 12/5$.
\end{proposition}

\begin{proposition}
\label{Phk=-4}For $x\in \left( 0,\pi /2\right) $, inequality 
\begin{equation}
2\left( \frac{\tanh x}{x}\right) ^{p}<\left( \frac{x}{\sinh x}\right) ^{4p}+1
\label{Ph-4}
\end{equation}%
holds if and only if $p<0$ or $p\geq 6/5.$
\end{proposition}

\end{document}